\documentclass[12pt, a4paper]{article}

\usepackage{pdfsync}
\usepackage{amsmath}\multlinegap=0pt
\usepackage[latin1]{inputenc}
\usepackage{amsthm}
\usepackage{amssymb}
\usepackage{hyperref}
\usepackage{graphicx}

\numberwithin{equation}{section}
\theoremstyle{plain}
\newtheorem{thm}{Theorem}[section]

\newtheorem{prop}[thm]{Proposition}
\newtheorem{lem}[thm]{Lemma}

\theoremstyle{definition}

\theoremstyle{remark}

\newcommand{\Lin}{\mathcal{L}}

\newcommand\N{{\mathbb N}}

\newcommand\pref[1]{(\ref{#1})}

\let \eps\varepsilon

\DeclareMathOperator{\id}{id}

\def\<#1,#2>{\left<#1,#2\right>}

\newcommand\FF{{\cal F}}

\newcommand\tT{{\widetilde {T}}}

\newcommand\tTet{\widetilde {\theta}}

\title{A differential approach to the multi-marginal Schr\"{o}dinger system}

\author{Guillaume Carlier\thanks{CEREMADE, UMR CNRS 7534, Universit\'e Paris IX
Dauphine, Pl. de Lattre de Tassigny, 75775 Paris Cedex 16, FRANCE and INRIA-Paris, MOKAPLAN,
\texttt{carlier@ceremade.dauphine.fr}}
\and
Maxime Laborde \thanks{Department of Mathematics and Statistics, McGill University, Montreal, CANADA,
\texttt{maxime.laborde@mcgill.ca}}
}

\begin{document}

\maketitle

\begin{abstract}
We develop an elementary and self-contained differential approach, in an $L^{\infty}$ setting, for well-posedness (existence, uniqueness and smooth dependence with respect to the data) for the multi-marginal Schr\"{o}dinger system which arises in the entropic regularization of optimal transport problems.

\end{abstract}


\medskip

\textbf{MS Classification:} 45G15, 49K40.

\section{Introduction}\label{sec-intro} 
 
 Multi-marginal optimal transport problems arise in various applied settings such as economics, quantum chemistry, Wasserstein barycenters... Contrary to the well-developed two-marginals theory (see the textbooks of Villani \cite{VilON, Villani-TOT2003} and  Santambrogio \cite{FilOptimal}), the structure of solutions of such problems is far from being well-understood in general. This explains the need for good numerical/approximation methods  among which the entropic approximation (which has its roots in the seminal paper of  Schr\"{o}dinger  \cite{schrodinger1932}) method plays a distinguished role both for its simplicity and its efficiency, see Cuturi \cite{Cut}, Benamou et al. \cite{Ben}. Roughly speaking, as its name indicates, the entropic approximation strategy consists in approximating the initial optimal transport problem by the minimization of a relative entropy with respect to the Gibbs kernel associated to the transport cost.  Rigorous $\Gamma$-convergence results as well as dynamic formulations for the quadratic transport cost were studied in particular by L\'eonard, see \cite{leonard2012schrodinger}, \cite{leonard2013survey} and the references therein.

 \smallskip 
 
 At least formally, joint measures that minimize a relative entropy subject to marginal constraints have a very simple structure, their density is the reference kernel multiplied by the tensor product of potentials (which we will call Schr\"{o}dinger potentials) which are  constrained by the prescribed marginal conditions.  However, the existence and regularity of Schr\"{o}dinger potentials cannot be taken for granted as a direct consequence of Lagrange duality because of constraints qualification issues.  The problem at stake is a system of nonlinear integral equations where the data are the kernel and the marginals and the unknowns are the Schr\"{o}dinger potentials. In the two-marginals case, there is a very elegant contraction argument for the Hilbert projective metric which shows the well-posedness of this system, see in particular Borwein, Lewis and Nussbaum \cite{BoLeNu}. This contraction argument is constructive and gives linear-convergence of the Sinkhorn algorithm which consists in solving alternatively the two integral equations of the system.   It is not obvious to us though whether this approach can be extended to the multi-marginal case (for which existence results exist but, apart from the case of finitely supported measures, rely on rather involved and abstract arguments, see for instance Borwein and Lewis \cite{BoLe}). Our goal is to give an elementary differential proof of the well-posedness of the Schr\"{o}dinger system in an $L^{\infty}$ setting.

 \smallskip
 
 This short paper is organized as follows. Section \ref{sec-prel} is devoted to  the presentation of the multi-marginal Schr\"{o}dinger system and its variational interpretation. Section \ref{sec-invert} deals with local invertibility whereas section \ref{sec-wellpos} is devoted to global invertibility and well-posedness. Section \ref{sec-furth} gives some further properties of the Schr\"{o}dinger map.

\section{Preliminaries}\label{sec-prel} 

\subsection{Data and assumptions}

We are given an integer $N\ge 2$, $N$ probability spaces $(X_i, \FF_i, m_i)$, $i=1, \ldots, N$ and set 
\begin{equation}\label{product}
X:=\prod_{i=1}^N X_i,  \FF:=\bigotimes_{i=1}^N \FF_i, \; m:=\bigotimes_{i=1}^N m_i.
\end{equation}
Given $i\in \{1, \ldots, N\}$,  we will denote by $X_{-i}:=\prod_{j \neq i}^N X_j$, $m_{-i}:=\bigotimes_{j\neq i}^N m_j$   and will always identify $X$ to $X_i\times X_{-i}$ i.e. will denote $x=(x_1, \ldots, x_N)\in X$ as $x=(x_i, x_{-i})$.

We shall denote by $L^{\infty}_{++}(X_i, \FF_i, m_i)$ (respectively $L^{\infty}_{++}(X, \FF, m)$) the interior of the positive cone of  $L^{\infty}(X_i, \FF_i, m_i)$ (respectively $L^{\infty}(X, \FF, m)$) and consider a kernel $K\in L^{\infty}_{++}(X, \FF, m)$ as well as densities $\mu_i \in L^{\infty}_{++}(X_i, \FF_i, m_i)$ with the same total mass:
\begin{equation}\label{balance}
\int_{X_i} \mu_i \mbox{d} m_i= \int_{X_j} \mu_j \mbox{d} m_j, \;   i, j \in \{1, \ldots, N\}.
\end{equation}
Our aim is to show the well-posedness of the multi-marginal Schr\"{o}dinger system: find   potentials $\varphi_i$ in $L^{\infty}(X_i, \FF_i, m_i)$ (called Schr\"{o}dinger potentials) such that for every $i$ and $m_i$-almost every $x_i\in X_i$ one has:
\begin{equation}\label{schrosys}
\mu_i(x_i)= e^{\varphi_i(x_i)} \int_{X_{-i}} K(x_i, x_{-i}) e^{\sum_{j\neq i} \varphi_j(x_j)}   \mbox{d} m_{-j}(x_{-j}). 
\end{equation}
Clearly if $\varphi=(\varphi_1, \ldots , \varphi_N)$ solves \pref{schrosys} so does every family of potentials of the form $(\varphi_1+\lambda_1, \ldots , \varphi_N+ \lambda_N)$ where the $\lambda_i$'s are constants with zero-sum, it is therefore natural to add as a normalization conditions  to \pref{schrosys} the additional $N-1$ linear equations:
\begin{equation}\label{normaliz}
\int_{X_i} \varphi_i \mbox{d} m_i=0, \; i=1, \ldots, N-1. 
\end{equation}

\subsection{Variational interpretation}

It is worth here recalling the origin of the Schr\"{o}dinger system in terms of minimization problems with multi-marginal constraints. Given $\mu=(\mu_1, \cdots, \mu_N)\in \prod_{i=1}^N  L^{\infty}_{++}(X_i, \FF_i, m_i)$ satisfying \pref{balance},  consider the entropy minimization problem
\begin{equation}\label{minient}
\inf_{q \in \Pi(\mu)} H(q\vert K m)
\end{equation}
where $\Pi(\mu)$ is the set of measures on $X$ having marginals $(\mu_1 m_1, \ldots, \mu_N m_N)$ (the nonemptyness of this set being guaranteed by \pref{balance}), $Km$ denotes the measure (equivalent to $m$) having density $K$ with respect to $m$ and $H$ denotes the relative entropy:
\[H(q\vert K m):=\begin{cases} \int_X \Big( \log\Big( \frac{1}{K} \frac{ d q}{d m} \Big) -1\Big) \mbox{d} q  \mbox{ if $q \ll m $}\\ + \infty \mbox{ otherwise}.   \end{cases} \]
A motivation for \pref{minient} is the following, when $K=e^{-\frac{c}{\eps}}$ is the Gibbs kernel associated to some cost function $c$ and $\eps>0$ is a small (temperature) parameter, then \pref{minient} is an approximation of the multi-marginal optimal transport problem which consists in finding a measure in $\Pi(\mu)$ making the average of the cost $c$ minimal (see \cite{leonard2012schrodinger}, \cite{leonard2013survey}, \cite{CDPS}).

\smallskip

At least formally, \pref{minient} is dual to the concave unconstrained maximization problem
\begin{equation}\label{dual}
\sup_{\varphi=(\varphi_1, \ldots, \varphi_N)} \sum_{i=1}^N \int_{X_i} \varphi_i \mu_i \mbox{d} m_i -\int_X K(x) e^{\sum_{j=1}^N \varphi_j(x_j)} \mbox{d} m(x)
\end{equation}
and if $\varphi\in \prod_{i=1}^N L^{\infty}(X_i, \FF_i, m_i)$ solves \pref{dual} (the point is that the existence of such a  minimizer cannot be taken for granted) it is a critical point of the (differentiable) functional in \pref{dual} which exactly leads to the Schr\"{o}dinger system \pref{schrosys}. Moreover interpreting such a $\varphi$ as a family of Lagrange multipliers associated to the marginal constraints in \pref{minient} leads to the guess that the solution $q$ of \pref{minient} should be of the form $q=\gamma m$ with a density kernel $\gamma$ of the form
\begin{equation}\label{ansatz}
\gamma(x_1, \ldots, x_N) =K(x_1, \ldots, x_N) e^{\sum_{j=1}^N \varphi_j(x_j)}
\end{equation}
and the requirement that $q\in \Pi(\mu)$ also leads to \pref{schrosys}.  Of course, by concavity,  if $\varphi$ is a bounded solution of \pref{schrosys} it is a maximizer of \pref{dual} and $q=\gamma m$ given by \pref{ansatz} solves \pref{minient}.

 \section{Local invertibility}\label{sec-invert}
 
 Let us define
 \[E:=\left\{\varphi:=(\varphi_1, \ldots, \varphi_N)\in \prod_{i=1}^N L^{\infty}(X_i, \FF_i, m_i) : \; \int_{X_i} \varphi_i \mbox{d} m_i=0, \; i=1, \ldots, N-1\right\}\]
 which, equipped with the $L^{\infty}$ norm, is a Banach space.  For $\varphi =(\varphi_1, \ldots, \varphi_N) \in  \prod_{i=1}^N L^{\infty}(X_i, \FF_i, m_i)$ define $T(\varphi)=(T_1(\varphi), \ldots, T_N(\varphi))\in  \prod_{i=1}^N L^{\infty}(X_i, \FF_i, m_i)$ by
 \begin{equation}
T_i(\varphi)(x_i):=\int_{X^{-i}} K(x_i, x_{-i}) e^{\sum_{j=1}^N \varphi_j(x_j) }\mbox{d} m_{-i}(x_{-i}).
 \end{equation}
 Note that  $T(E)=T(\prod_{i=1}^N L^{\infty}(X_i, \FF_i, m_i) )\subset F_{++}$ where 
 \begin{equation}\label{F++}
 F_{++}:=F\cap  \prod_{i=1}^N L^{\infty}_{++}(X_i, \FF_i, m_i), 
 \end{equation}
and 
 \begin{equation}\label{F}
F:=\left\{\mu \in \prod_{i=1}^N L^{\infty}(X_i, \FF_i, m_i)\; : \int_{X_1} \mu_1 dm_1=\ldots=\int_{X_N} \mu_N dm_N \right\}. 
\end{equation} 
It will also be convenient to define the map $\tT=(\tT_1, \ldots, \tT_N)$ by $\tT_i(\varphi):=\log(T_i(\varphi))$ for $\varphi =(\varphi_1, \ldots, \varphi_N) \in  \prod_{i=1}^N L^{\infty}(X_i, \FF_i, m_i)$ i.e.
\begin{equation}\label{deftT}
\tT_i(\varphi)(x_i):= \varphi_i(x_i) + \log\Big(\int_{X^{-i}} K(x_i, x_{-i}) e^{\sum_{j\neq i} \varphi_j(x_j) }\mbox{d} m_{-i}(x_{-i})  \Big).
\end{equation}

Let us then observe that both $\tT$ and $T$ are of class $C^{\infty}$, more precisely for $\varphi$ and $h$ in  $\prod_{i=1}^N L^{\infty}(X_i, \FF_i, m_i)$, we have
 \[\tT'_i(\varphi)(h)(x_i)=h_i(x_i) + \frac{ \int_{X_{-i}} K(x_i, x_{-i}) e^{\sum_{k\neq i} \varphi_k(x_k) } \sum_{j\neq i} h_j(x_j) \mbox{d} m_{-i}(x_{-i})   }{\int_{X_{-i}} K(x_i, x_{-i}) e^{\sum_{j\neq i} \varphi_j(x_j) }\mbox{d} m_{-i}(x_{-i}) }  \]
 and
 \begin{equation}\label{manu}
 T'_i(\varphi)(h)(x_i)=e^{\tT_i(\varphi)(x_i)} \tT'_i(\varphi)(h)(x_i).
 \end{equation}
 
Let us fix $\varphi:=(\varphi_1, \ldots, \varphi_N)\in E$, observe that $\tT'(\varphi)$ extends (and we still denote by $\tT'(\varphi)$ this extension) to a bounded linear self map of  $\prod_{i=1}^N L^{2}(X_i, \FF_i, m_i)$ which is of the form
\begin{equation}
\tT'(\varphi):=\id+ L
\end{equation}
with $L$ a \emph{compact}  linear self map of  $\prod_{i=1}^N L^{2}(X_i, \FF_i, m_i)$. We then have the following:

\begin{prop}\label{local}
Let $\varphi \in E$ then $T'(\varphi)$ is an isomorphism between $E$ and $F$. In particular,  $T$ is a local $C^{\infty}$ diffeomorphism between $E$ and $F$, and $T(E)$ is open in $F_{++}$. 

\end{prop}

\begin{proof}
In view of \pref{manu}, the desired invertibility claim amounts to show that $\tT'(\varphi)$ is an isomorphism between $E$ and $F_{\varphi}$ the linear subspace of codimension $N-1$  consisting of  $\theta=(\theta_1, \ldots, \theta_N)\in \prod_{i=1}^N L^{\infty}(X_i, \FF_i, m_i)$ which satisfy
\begin{equation}\label{normavecfie}
\int_{X_1} e^{\tT_1(\varphi)}\theta_1 \mbox{d} m_1=\ldots=\int_{X_N} e^{\tT_N(\varphi)}\theta_N \mbox{d} m_N.
\end{equation}
Let us also denote by $F_{\varphi, 2}$ the set  of all $\theta=(\theta_1, \ldots, \theta_N)\in \prod_{i=1}^N L^{2}(X_i, \FF_i, m_i)$ which satisfy \pref{normavecfie}.

As noted above, one can write $\tT'(\varphi)=\id+L$ on $\prod_{i=1}^N L^{2}(X_i, \FF_i, m_i)$ with $L$ compact. Let us define the probability measure $Q_{\varphi}$ on $X$ given by
\begin{equation}\label{defQfi}
Q_{\varphi}(\mbox{d} x)=\frac{ K(x) e^{\sum_{j=1}^N \varphi_j(x_j)}  m (\mbox{d} x)}{ \int_X K(x) e^{\sum_{j=1}^N \varphi_j(x_j)} \mbox{d} m(x)}.
\end{equation}
For $i=1, \ldots, N$, let us now disintegrate $Q_{\varphi}$ with respect to its $i$-th marginal $Q_{\varphi}^i$:
\begin{equation}\label{dezinglrm}
Q_{\varphi}( \mbox{d} x_i,  \mbox{d} x_{-i}) = Q_{\varphi}^{-i}( \mbox{d} x_{-i} \vert x_i) \otimes Q_{\varphi}^i(\mbox{d} x_i) 
\end{equation}
where $Q_{\varphi}^{-i}( \mbox{d} x_{-i} \vert x_i)$ is the conditional probability of $x_{-i}$ given $x_i$ according to $Q_{\varphi}$. The compact operator $L$ can then conveniently be expressed in terms of the corresponding conditional expectations operators. Indeed,  setting $L(h)=(L_1(h), \ldots, L_N(h))$, we obviously have
\[L_i(h)(x_i)=\int_{X_{-i}}  \Big( \sum_{j\neq i}  h_j(x_j) \Big) Q^{-i}_{\varphi}( \mbox{d} x_{-i} \vert x_i) \mbox{  for $m_i$-a.e. $x_i\in X_i$}.\]
Let $h\in \prod_{i=1}^N L^{2}(X_i, \FF_i, m_i)$ be such that   $\tT'(\varphi)(h)=0$ (equivalently $T'(\varphi)(h)=0$) i.e. for every $i$ and $m_i$-a.e. $x_i\in X_i$, there holds
\[h_i(x_i)=-\int_{X_{-i}}  \Big( \sum_{j\neq i}  h_j(x_j) \Big) Q^{-i}_{\varphi}( \mbox{d} x_{-i} \vert x_i)\]
multiplying by $h_i(x_i)$ and then integrating with respect to $Q_{\varphi}^i$ gives
\[ \int_{X_i} h_i^2(x_i) \mbox{d} Q_{\varphi}^i( x_i)=- \sum_{j,\, j\neq i} \int_X h_i(x_i) h_j(x_j) \mbox{d} Q_{\varphi}(x)\]
summing over $i$ thus yields
\[\begin{split}
\int_X \Big(\sum_{i=1}^N h_i(x_i)\Big)^2  \mbox{d} Q_{\varphi}( x)&= \sum_{i=1}^N  \int_{X_i} h_i^2(x_i) \mbox{d} Q_{\varphi}^i( x_i)
+  \sum_{i,j,\,j\neq i} \int_X h_i(x_i) h_j(x_j) \mbox{d} Q_{\varphi}(x)\\
&=0.
\end{split}\]
Since $Q_{\varphi}$ is equivalent to $m$, we deduce that $\sum_{i=1}^N h_i(x_i)=0$ $m$-a.e. that is $h$ is constant and its components sum to $0$. Hence $\ker(\tT'(\varphi))$ has dimension $N-1$ and $\ker(\tT'(\varphi))\cap E=\{0\}$ i.e. $\tT'(\varphi)$ is one to one on $E$.

\smallskip

Since $L$ is a compact operator of $L^2$ and $\ker(\id+L)$ has dimension $N-1$, it follows from the Fredholm alternative Theorem (see chapter VI of \cite{brezis}) that $R(\id+L)$ has codimension $N-1$. Differentiating the relation
\[\int_{X_i} e^{\tT_i(\varphi)} \mbox{d} m_i=\int_{X_j} e^{\tT_j(\varphi)} \mbox{d} m_j, \; i, j \in \{1, \ldots, N-1\}\]
gives 
\[\int_{X_i} e^{\tT_i(\varphi)}  \tT'_i(\varphi) (h) \mbox{d} m_i=\int_{X_j} e^{\tT_j(\varphi)} \tT'_j(\varphi) (h)   \mbox{d} m_j, \; i, j \in \{1, \ldots, N-1\}\]
i.e. $\tT'(\varphi)(h) \in F_{\varphi}$ for every $h\in  \prod_{i=1}^N L^{\infty}(X_i, \FF_i, m_i)$. Likewise, we also have $\tT'(\varphi)(h)\in F_{\varphi,2}$, for every  $h\in  \prod_{i=1}^N L^{2}(X_i, \FF_i, m_i)$.   Since $F_{\varphi,2}$ has codimension $N-1$, we  get
\begin{equation}
R(\id+L)=\tT'(\varphi) \Big( \prod_{i=1}^N L^{2}(X_i, \FF_i, m_i)\Big)=F_{\varphi,2}.
\end{equation}
In particular, for every $\theta\in F_{\varphi}$ there exists $h\in \prod_{i=1}^N L^{2}(X_i, \FF_i, m_i)$ such that $\theta= h+ L(h)$ since obviously $L$ maps $\prod_{i=1}^N L^{2}(X_i, \FF_i, m_i)$ into $\prod_{i=1}^N L^{\infty}(X_i, \FF_i, m_i)$ we have $h\in \prod_{i=1}^N L^{\infty}(X_i, \FF_i, m_i)$. Finally, since $\tT'(\varphi)(h)=\tT'(\varphi)(\widetilde{h})$ whenever $h-\widetilde{h}$ is a vector of constants summing to zero, we may also assume that $h\in E$. This shows that $\tT'(\varphi)(E)=F_{\varphi}$ or equivalenty $T'(\varphi)(E)=F$.

\smallskip

We have shown that $T'(\varphi)$ is an isomorphism between the Banach spaces $E$ and $F$, the local invertibility claim thus directly follows from the local inversion Theorem.

\end{proof}

 \section{Global invertibility and well-posedness}\label{sec-wellpos}
 
 To pass from local to global invertibiliy of $T$, we invoke classical arguments \`a la Caccioppoli-Hadamard (see for instance \cite{KP}). First of all, it is easy to see  that $T$ is one to one on $E$:

 \begin{prop}\label{onetoone}
 The map $T$ is injective on $E$.
 
 \end{prop}
 
 \begin{proof}
 If $\varphi$ and $\psi$ are in $E$ and $T(\varphi)=T(\psi):=\mu$, then both $\varphi$ and $\psi$ are solutions of the maximization problem \pref{dual}, since the functional in  \pref{dual} is the sum of a linear term and a term that is strictly concave in the direct sum of the potentials we should have $\sum_{i=1}^N \varphi_i(x_i) =\sum_{i=1}^N \psi_i(x_i)$ which by the normalization condition in the definition  of $E$ implies that $\varphi=\psi$.

 \end{proof}
 
 Next we observe that:
 
 \begin{lem}\label{proper}
 $T(E)$ is closed in $F_{++}$. 
  \end{lem}
 
 \begin{proof}
 Let $(\varphi^n)_n\in E^{\N}$ be such that $\mu^n:=T(\varphi^n)$ converges in $L^{\infty}$ to some $\mu\in F_{++}$. Let $\psi^n=(\varphi_1^n+ \lambda_1^n, \ldots, \varphi_N^n+\lambda_N^n)$ where the $\lambda_i^n$'s are constant which sum to zero and chosen in such a way that
 \begin{equation}\label{altnormaliz}
 \int_{X_i} e^{\psi_i^n} \mbox{d} m_i=1, \; i=1, \ldots, N-1,
 \end{equation}
 this ensures that  $\mu^n:=T(\psi^n)$ i.e. for every $i$ and $m_i$- a.e. $x_i\in X_i$
 \begin{equation}\label{benallaenculemanu}
 \log(\mu_i^n(x_i))=\psi_i^n(x_i)+ \log\Big( \int_{X_{-i}} K(x_i, x_{-i}) q_{-i}^n (x_{-i}) \mbox{d} m_{-i}(x_{-i})   \Big)
 \end{equation}
 where
 \[q_{-i}^n(x_{-i}):=e^{\sum_{j\neq i} \psi_j^n(x_j)}.\]
 Since $(\mu_N^n)_n$ is uniformly bounded and bounded away from $0$ and so is $K$, we deduce that $(e^{\psi_N^n})_n$ is bounded and bounded away from $0$ in $L^{\infty}$ i.e. $(\psi_N^n)_n$ is bounded in $L^{\infty}(X_N, \FF_N, m_N)$. From this $L^{\infty}$ bound on $(\psi_N^n)_n$, the fact that $K\in L_{++}^{\infty}(X, \FF, m)$ and the uniform bounds from above and from below on $\mu_i^n$, we deduce that $\psi_i^n$ is bounded in $L^\infty$ for $i=1, \ldots, N-1$. In particular, taking subsequences if necessary, we may assume that  for every $i$, $(q_{-i}^n)_n$ converges weakly $*$ in $L^\infty(X_{-i}, \FF_{-i}, m_{-i})$ to some $q_{-i}$, in particular $\int_{X_{-i}} K(x_i, x_{-i}) q_{-i}^n (x_{-i}) \mbox{d} m_{-i}(x_{-i})$ converges for $m_i$-a.e. $x_i$ to $\int_{X_{-i}} K(x_i, x_{-i}) q_{-i} (x_{-i}) \mbox{d} m_{-i}(x_{-i})$. But since $\log(\mu_i^n)$ converges in $L^{\infty}(X_i, \FF_i, m_i)$ to $\log(\mu_i)$, we deduce from \pref{benallaenculemanu} that $\psi_i^n$ converges $m_i$-a.e. (and also in $L^p$ for every $p\in [1, +\infty)$ by Lebesgue's dominated convergence Theorem) to some $\psi_i \in L^{\infty}$. Passing to the limit in  \pref{benallaenculemanu}, we then have $\mu=T(\psi)$ or equivalently $\mu=T(\varphi)$ for $\varphi \in E$ such that $\varphi-\psi$ is constant. This shows that $T(E)$ is closed in $F_{++}$.

 \end{proof}

 We are now in position to state our main result:
 
 \begin{thm}\label{invglothm}
 
 For every $\mu\in F_{++}$, the multi-marginal Schr\"{o}dinger system \pref{schrosys} admits a unique solution $\varphi=S(\mu)\in E$, moreover  $S \in C^{\infty}(F_{++}, E)$.

 \end{thm}

 \begin{proof}
 It follows from Proposition \ref{local} that $T(E)$ is open in $F_{++}$ and Lemma \ref{proper} ensures it is closed in $F_{++}$, since $F_{++}$ is connected (it is actually convex) we deduce that $T(E)=F_{++}$. Together with Proposition \ref{onetoone} this implies that $T$ is a bijection between $E$ and $F_{++}$, the smoothness claim then follows from Proposition \ref{local}.  
 \end{proof}

\section{Further properties of the Schr\"{o}dinger map}\label{sec-furth} 

From now on, we refer to  the smooth map $S=T^{-1}$ : $F_{++}  \to E$ from Theorem \ref{invglothm} as the Schr\"{o}dinger map. Our aim now is to study the (local) Lipschitz behavior of $S$. 
 Given $M  \ge 1$ we define
\begin{equation}
F_{++, M}:= \{\mu \in F_{++} \: : \; \frac{1}{M} \le \mu_i \le M \mbox{ $m_i$-a.e.}\}. 
\end{equation}

Let us start with an elementary a priori bound:

\begin{lem}\label{collombpd}
For every $M\ge 1$ there is a constant $R_{M}$ such that $S(F_{++,M})$ is included in the ball of radius $R_M$ of   $\prod_{i=1}^N L^{\infty}(X_i, \FF_i, m_i)$.
\end{lem}

\begin{proof}
Let $\mu\in F_{++, M}$ and $\varphi=S(\mu)$, as in the proof of Lemma \ref{proper} we introduce  constants $\lambda_i$ with zero sum such that $\mu=T(\psi)$ with $\psi_i=\varphi_i+\lambda_i$ is normalized by \pref{altnormaliz} (instead of \pref{normaliz}). Using the fact that $K$ is bounded and bounded away from $0$, that $M^{-1}  \le \mu_N \le  M$, \pref{altnormaliz} and $\mu_N=T_N(\psi)$ gives upper and lower bounds on $e^{\psi_N}$ i.e. an $L^{\infty}$ bound (depending on $M$ and $K$ only) on $\psi_N$.  This bound and  $\mu_i =T_i(\psi)$ in turn provide $L^{\infty}$ bounds on $\psi_i$ for $i=1, \ldots, N-1$. Finally, we get bounds on the constants $\lambda_i$ since $\lambda_i=\int_{X_i} \psi_i \mbox{d} m_i$ for $i=1, \ldots, N-1$ and $\lambda_N=-\sum_{i=1}^{N-1} \lambda_i$.  This gives the desired bounds on $\varphi=S(\mu)$.

\end{proof}

More interesting in possible applications, is the Lipschitz behavior of $S$ given by the folllowing 

\begin{thm}\label{lipS}
For every $M \ge 1$ there is a constant $C_M$, such that\footnote{In formulas \pref{l2lip} (respectively \pref{linflip}) $L^2$ (resp. $L^{\infty}$) is a abbreviated notation for $\prod_{i=1}^N L^{2}(X_i, \FF_i, m_i)$ (resp. $\prod_{i=1}^N L^{\infty}(X_i, \FF_i, m_i))$.} for every $\mu$ and  $\nu$ in $F_{++, M}$, there holds
\begin{equation}\label{l2lip}
\Vert S(\mu)-S(\nu)\Vert_{L^2} \le C_{M} \Vert \mu-\nu \Vert_{L^2},
\end{equation}
and
\begin{equation}\label{linflip}
\Vert S(\mu)-S(\nu)\Vert_{L^{\infty}} \le C_{M}   \Vert \mu-\nu  \Vert_{L^{\infty}}.
\end{equation}
\end{thm}

\begin{proof}
Let $\mu \in F_{++,M}$ and $\varphi=S(\mu)\in E$, our aim is to estimate the operator norm of $S'(\mu)=[T'(\varphi)]^{-1}$ (first in $L^2$ and then in $L^{\infty}$). Let  $\theta\in F$ and $h=S'(\mu)\theta$ i.e. $T'(\varphi)h=\theta$ which can be rewritten as
\begin{equation}
\tT'_i(\varphi) h=\tTet_i \mbox{ with } \tTet_i:=\frac{\theta_i}{\mu_i}.
\end{equation}

Defining the measure $Q_{\varphi}$ by \pref{defQfi} and disintegrating it with respect to its $i$-th marginal as in \pref{dezinglrm} in the proof of proposition \ref{local} gives that for every $i$ and $m_i$-a.e. $x_i$ one has
\begin{equation}\label{macronconnar}
\tTet_i(x_i)= h_i(x_i) + \int_{X_{-i}}  \Big( \sum_{j\neq i}  h_j(x_j) \Big) Q^{-i}_{\varphi}( \mbox{d} x_{-i} \vert x_i).
\end{equation}
We then argue in a similar way as we did in the proof of proposition \ref{local},  multiplying \pref{macronconnar} by $h_i$ and integrating with respect to $Q_{\varphi}^i$ and summing over $i$, we obtain
\begin{equation}\label{controll2}
\sum_{i=1}^N \int_{X_i} \tTet_i (x_i)h_i(x_i) \mbox{d} Q^i_{\varphi} (x_i)= \int_{X} \Big( \sum_{j=1}^N  h_j(x_j) \Big)^2 \mbox{d} Q_{\varphi} (x). 
\end{equation}
Next we observe that thanks to the fact that $\mu\in F_{++,M}$, the upper and lower bounds on $K$ and Lemma \ref{collombpd} there is a constant $\nu_M\ge 1$ such that 
\begin{equation}\label{equivprob}
 \frac{ m}{\nu_M}  \le  Q_{\varphi} \le \nu_M m, \;  \frac{ m_i}{\nu_M} \le  Q^i_{\varphi} \le \nu_M m_i.
\end{equation}
Using the fact that $\Vert \tTet_i \Vert_{L^2(X_i, \FF_i, m_i)} \le M \Vert \theta_i \Vert_{L^2(X_i, \FF_i, m_i)}$, \pref{equivprob} and Cauchy-Schwarz inequality, we deduce from \pref{controll2} that there is a constant $C_M$ such that
\begin{equation}
\int_{X} \Big( \sum_{j=1}^N  h_j(x_j) \Big)^2 \mbox{d} m (x) \le C_M \sum_{i=1}^N \Vert \theta_i \Vert_{L^2(X_i, \FF_i, m_i)}  \Vert h_i \Vert_{L^2(X_i, \FF_i, m_i)}.
\end{equation}
Finally recall that since $h\in E$ we have
\[\int_{X} \Big( \sum_{j=1}^N  h_j(x_j) \Big)^2 \mbox{d} m (x)= \sum_{j=1}^N \int_{X_j} h_j^2(x_j) \mbox{d} m_j(x_j) =:\Vert h\Vert^2_{L^2}\]
hence 
\begin{equation}\label{brizitapoil}
 \Vert h \Vert_{L^2}=\Vert S'(\mu) \theta \Vert_{L^2} \le C_M \Vert  \theta \Vert_{L^2} \mbox{ i.e. } \sup_{\mu\in F_{++, M}}  \Vert S'(\mu)\Vert_{\Lin(L^2)} \le C_M.
\end{equation}
By the mean-value inequality \pref{brizitapoil} immediately gives the Lipschitz in $L^2$ estimate \pref{l2lip}.

\smallskip

As for a bound on the operator norm of $S'(\mu)$ in $L^\infty$ , we first observe that for some positive constant $\lambda_M$ we have $Q_{\varphi}^{-i} \le \lambda_M m_{-i}$, so that \pref{macronconnar} gives
\[\begin{split}
 \Vert h_i\Vert_{L^{\infty}} &\le \Vert \tTet_i\Vert_{L^{\infty}(m_i)}  +  \lambda_M \sum_{j\ne i} \int_{X_j} \vert h_j(x_j)\vert \mbox{d} m_j(x_j) \\
& \le  M   \Vert \theta_i\Vert_{L^{\infty}(m_i)}+ \lambda_M \sqrt{N} \Vert h\Vert_{L^2} \\
 & \le M   \Vert \theta_i\Vert_{L^{\infty}(m_i)}+ \lambda_M   \sqrt{N} C_M \Vert \theta \Vert_{L^2}\\
 & \le C'_M \Vert \theta\Vert_{L^{\infty}}
 \end{split}\]
where we have used Cauchy-Schwarz inequality in the second line and \pref{brizitapoil} in the third one. This clearly implies \pref{linflip}. 

\end{proof}




 
 
{ \bf{Acknowledgments:}}   G.C. is grateful to the Agence Nationale de la Recherche for its support through the project MAGA (ANR-16-CE40-0014).

\bibliographystyle{plain}

\bibliography{bibli}

\end{document}